\documentclass[11pt,reqno]{article}
\usepackage{bbding}
\usepackage{pifont}
\usepackage{amsmath}
\usepackage{mathrsfs}
\usepackage{amsfonts,bm}
\usepackage{amsthm}
\usepackage{epsfig}
\usepackage[]{harpoon}
\usepackage[active]{srcltx}
\usepackage{indentfirst,latexsym}
\usepackage{psfrag}
\usepackage[]{amssymb}
\usepackage{cite}
\usepackage{caption}

\newtheorem{theorem}{Theorem}[section]
\newtheorem{lemma}[theorem]{Lemma}

\newcommand{\pp}{\partial_2}
\newcommand{\pn}{\partial_1}

\newcommand{\new}{\newcommand*}\new{\rnew}{\renewcommand*}
\new{\newe}{\newenvironment*}\new{\stl}{\setlength}
\stl{\textwidth}{155mm}\stl{\textheight}{22cm}\stl{\headheight}{0cm}
\stl{\topmargin}{0cm}\stl{\oddsidemargin}{0.5cm}\stl{\evensidemargin}{0cm}
\rnew{\arraystretch}{1.1}\rnew{\baselinestretch}{0.95}
\renewcommand{\thefootnote}{\ding{73}}
\newtheorem{thm}{Theorem}

\newtheorem{lem}{Lemma}

\newtheorem{rem}{Remark}

\newcommand{\fr}{\frac}

\newcommand{\pa}{\partial}

\numberwithin{equation}{section}
\newe{keywords}
   {\begin{quote}{\bf Keywords:}}
      {\end{quote}}
\newe{AMS}
   {\begin{quote}{\bf AMS subject classification 2020:}}
      {\end{quote}}
\newe{MSC}{\vspace*{5mm}
    {\noindent\it Mathematics Subject Classification(2020):}}{}
\new{\sect}[1]{\section{#1}\setcounter{equation}{0}
 \setcounter{thm}{0}\setcounter{lmm}{0}\setcounter{rmk}{0} }

\begin{document}

\title{Singularity formation for the supersonic inward wave of compressible Euler equations with radial symmetry}

\author{
Geng Chen\thanks{Department of Mathematics,
University of Kansas, Lawrence, KS
66045 ({\tt gengchen@ku.edu}).}
\and
Faris A. El-Katri\thanks{Department of Mathematics,
University of Kansas, Lawrence, KS
66045 ({\tt elkatri@ku.edu}).}
\and
Yanbo Hu\thanks{Department of Mathematics, Zhejiang University of Science and Technology, Hangzhou 310023, PR China ({\tt yanbo.hu@hotmail.com}).}
\and
Yannan Shen\thanks{Department of Mathematics,
University of Kansas, Lawrence, KS
66045 ({\tt yshen@ku.edu}).}
}

\rnew{\thefootnote}{\fnsymbol{footnote}}

\date{}

\maketitle
\begin{abstract}

In this paper, we consider the singularity formation of smooth solutions for the compressible radially symmetric Euler equations. By applying the characteristic method and the
invariant domain idea, we show that, for polytropic ideal gases with $\gamma\geq3$, the smooth solution develops a singularity in finite time for a class of initial supersonic inward waves.
\end{abstract}

\begin{keywords}
Compressible Euler equations, radially symmetric flow, supersonic inward wave, singularity formation
\end{keywords}

\begin{AMS}
76N15, 35L65, 35L67
\end{AMS}

\section{Introduction}\label{S1}

The compressible radially symmetric isentropic Euler equations with $\gamma$-law pressure read that \cite{courant}
\begin{align}\label{1.1}
\begin{split}
(r^m\rho)_t+(r^m\rho u)_r=&\ 0, \\
(r^m\rho u)_t+(r^m\rho u^2)_r+ r^m p_r=&\ 0, \\
p=&\ K\rho^\gamma.
\end{split}
\end{align}
Here $t\geq0$ denotes the time coordinate, $r>0$ denotes the spatial coordinate, and $m\geq1$ is an integer. Particularity, $m=1, 2$ for flows with cylindrical or spherical symmetry, respectively.
The unknown variables have their ordinary meaning: $\rho(r,t)$ is the density, $u(r,t)$ is the particle velocity, $p(r,t)$ is the pressure and $\gamma>1$ is the adiabatic constant.

As one of the most representative models of nonlinear hyperbolic
conservation laws, the compressible Euler equations have always received great attention. The radially symmetric flow is one important flow in compressible fluid mechanics, as it occurs in many meaningful physical situations \cite{courant, Whi}. The study of the radially symmetric Euler equations has a long history, and the demonstration and construction of singularity through its classical solutions is one of the key concerns of numerous researchers. The singularity formation of smooth solution in finite time to the $2\times2$ reducible homogeneous hyperbolic systems has been widely discussed by using Lax's framework \cite{lax2} based on the characteristic method, see for example \cite{Ali, John, Kong, Liu1} for the early works.
For this type of system, the equations of Riemann invariants are homogeneous, and so the signs of the gradients on the Riemann invariants can be easily determined by initial data. For the one-dimensional isentropic Euler equations with $\gamma\geq3$, one can directly apply Lax's framework to establish
a complete dichotomy result, that is, the finite time singularity forms if and only if the initial data include compression. For the case $1<\gamma<3$, a suitable density lower bound estimate
is first needed to overcome the degeneracy of the Riccati equations caused by the degenerate of density.
In a series of papers \cite{G9, G10, CCZ, CPZ, CY, CYZ}, the first author derived the optimal time-dependent lower bound on density and established a sequence of fairly complete results on the singularity formation theory for the isentropic and non-isentropic Euler equations in one space dimension.

We also refer the reader to methods that study the singularity formation of the compressible Euler equations in multiple space dimensions, see Sideris's framework \cite{sideris1, sideris2}, the geometric framework \cite{Ch1, Ch2} etc.
There are also many recent results on the construction of shock formation that accurately describe the blowup process of the gradient of the solution, see among others \cite{Vicol1, Vicol2, Vicol3, Luk1, Luk2, Vicol5, Vicol4}.

For the radially symmetric Euler equations \eqref{1.1}, it is natural to mimic the one-dimensional case to utilize the gradient of Riemann variables. However, due to the influence of non-homogeneous terms, such Riccati system is non-homogeneous, making it complex to establish the desired invariant region then singularity formation results.
In \cite{CCW}, Cai, Chen, and Wang have been adopting such kind of Riccati system to investigate
the singularity formation for supersonic expanding wave of \eqref{1.1} with $1<\gamma<3$ by cumbersomely analyzing the complicated non-homogeneous Riccati equations. A solution $(\rho,u)$ of \eqref{1.1} is called a supersonic expanding wave if it satisfies $u>h>0$, where $h=\sqrt{p'(\rho)}$ is the local sound speed. The supersonic expanding requirement can be directly guaranteed by the invariant region of Riemann variables, which ensures the strict positivity of the two eigenvalues of \eqref{1.1} and
the boundedness of the coefficients of the Riccati equations. In a recent paper, we \cite{CEHS} introduced a pair of accurate gradient variables (called rarefaction and compression characters) for the supersonic expanding wave of \eqref{1.1} with $1<\gamma<3$ to study the global existence and singularity formation of smooth solutions. In terms of the defined rarefaction and compression characters, the corresponding Riccati equations are homogeneous, from which we demonstrated that smooth solutions with
rarefactive initial data exist global-in-time, while singularity forms in finite time when the initial data include strong compression somewhere. The idea in \cite{CEHS} comes from the fact that the stationary solution of \eqref{1.1} is neither rarefactive nor compressive. Applying the same idea, the results in \cite{CEHS} were
subsequently extended to the radially symmetric relativistic Euler equations \cite{CGH}.

In the current paper, we are interested in the singularity formation for the supersonic inward wave of the radially symmetric Euler equations \eqref{1.1}. A solution of \eqref{1.1} is called a supersonic inward wave if it satisfies $u<-h<0$. We explore the formation of singularities of smooth supersonic inward solutions in finite time. The motivation for studying this problem comes from at least two aspects. First, the singularity formation of supersonic inward waves has important significance in physics. Second, for $\gamma\geq3$, the invariant region of the expanding waves of \eqref{1.1} cannot be persisted as the eigenvalues change sign, that is an expanding wave may transform into an inward wave. For the case of expanding waves, as handled in previous papers \cite{CCW, CEHS, CGH}, one can directly employ the invariant region of Riemann variables to establish the persistence of supersonic expansion properties by the appropriate initial data. However, it is ineffective for the case of inward waves because one of the Riemann variables is strictly monotonically increasing along the characteristic, see \eqref{2.4} below, which will inevitably disrupt the supersonic inward properties. This is the main difficulty in studying the singularity formation for the supersonic inward waves.

In this paper, we consider the supersonic inward wave of system \eqref{1.1} with $\gamma\geq3$. We  still use the key gradient variables $(\alpha,\beta)$ in \cite{CEHS}. One key observation of this paper is to simplify the homogeneous Riccati equations in \cite{CEHS} on $(\alpha,\beta)$ into \eqref{2.6}. Then we introduce the weighted gradient variables $(\tilde{\alpha},\tilde{\beta})$, see \eqref{a3}, and establish a suitable invariant region for the five variables $\rho, u, \tilde{\alpha}, \tilde{\beta}$ and $u+h$ simultaneously. Specifically, when deriving the strict negative upper bound of $u+h$, we abandon the equation of Riemann variable but instead adopt its equation along the 1-characteristic, see \eqref{2.10} below. This treatment allows us to rely on the positivity of $\tilde{\alpha}$ to obtain the negativity of $u+h$ within the considered time. We show that the gradient variable $\tilde{\beta}$ develops a singularity in finite time when the initial value $-\tilde{\beta}|_{t=0}$ is sufficiently large at somewhere. The main conclusion of the paper is presented in Theorem \ref{thm1} in Section \ref{S3}.

\section{The characteristic equations}\label{S2}

We introduce the sound speed
$$
h=\sqrt{p_\rho}=\sqrt{K\gamma}\,\rho^\frac{\gamma-1}{2},
$$
as the variable to take the place of $\rho$.
For smooth solutions, system \eqref{1.1} can be written as
\begin{align}\label{2.1}
\begin{split}
h_t+uh_r+\fr{\gamma-1}{2}hu_r=& -\fr{\gamma-1}{2}\frac{m uh}{r}, \\
u_t+uu_r+\fr{2}{\gamma-1}hh_r=&\ 0,
\end{split}
\end{align}
with characteristic speeds
\begin{align}\label{2.2}
c_1=u-h,\quad c_2=u+h.
\end{align}
We further introduce the Riemann variables $w$ and $z$
\begin{align}\label{a1}
w=u+\frac{2}{\gamma-1}h,\quad z=u-\frac{2}{\gamma-1}h.
\end{align}
Denote the operators
\begin{align}\label{2.3}
\pa_1=\pa_t+c_1\pa_r,\quad \pa_2=\pa_t+c_2\pa_r.
\end{align}
We then obtain the governing equations of $(w,z)$ as the following form
\begin{align}\label{2.4}
\pa_2w=-\fr{m}{r}uh,\quad  \pa_1z=\fr{m}{r}uh.
\end{align}

Following our earlier paper \cite{CEHS}, we adopt the following pair of gradient variables
\begin{align}\label{2.5}
\begin{split}
\alpha=&-\frac{\partial_1(r^m \rho u)}{r^m \rho c_2}=u_r +\fr{2}{\gamma-1}h_r+\frac{mhu}{rc_2}, \\
\beta=&-\frac{\partial_2(r^m \rho u)}{r^m \rho c_1}=u_r -\fr{2}{\gamma-1}h_r-\frac{mhu}{rc_1}.
\end{split}
\end{align}
and the following lemma
\begin{lemma}[Lemmas 3.1 \cite{CEHS}]
\label{lemma_ric}
For smooth solution of \eqref{1.1}, we have the following Riccati type equations on $\alpha$ and $\beta$ defined in \eqref{2.5}, when $r>0$ and $c_1c_2\neq 0$,
\begin{equation}\label{beta_eq}
\pn\beta=-\frac{1+\gamma}{4}\beta^2-\frac{3-\gamma}{4}\alpha\beta+A_1\alpha-B_1\beta,
\end{equation}
and \begin{equation}\label{alphacd1}
\pp\alpha=-\frac{\gamma+1}{4}\alpha^2-\frac{3-\gamma}{4}\alpha\beta+A_2\beta-B_2\alpha,
\end{equation}
where
\begin{equation}\label{B1Def}
B_1=\frac{m}{rc_1^2}\bigg(\frac{\gamma-1}{4}u^3-\frac{1}{2}h^3
-\frac{\gamma-1}{4}u^2h+
\frac{1}{2}uh^2+
\frac{hu{c_1}}{{c_2}}(h+\frac{\gamma-1}{2}u)\bigg),
\end{equation}
\begin{equation}\label{alphacd3}
B_2=\frac{m}{rc_2^2}\bigg(\frac{\gamma-1}{4}u^3+\frac{1}{2}h^3+
\frac{\gamma-1}{4}u^2h+
\frac{1}{2}uh^2+\frac{hu{c_2}}{{c_1}}(h-\frac{\gamma-1}{2}u)\bigg),
\end{equation}
and
\begin{align}\label{a2}
A_1=\fr{mc_2}{2rc_{1}^2}\bigg(\fr{\gamma-1}{2}u^2-h^2\bigg),\quad A_2=\fr{mc_1}{2rc_{2}^2}\bigg(\fr{\gamma-1}{2}u^2-h^2\bigg).
\end{align}
\end{lemma}
\begin{rem}
It is noted that \eqref{beta_eq} and \eqref{alphacd1} are homogeneous, that is, their right-hand sides vanish when $\alpha$ and $\beta$ both vanish. This crucial property is based on the special construction of $(\alpha,\beta)$ in \eqref{2.5}, which essentially originates from the fact that the stationary solution of \eqref{1.1} is neither rarefactive nor compressive.
\end{rem}
\bigskip

Here we notice that, by simple calculations, the coefficients $B_1$ and $B_2$ can be simplified, such
that the Riccati equations of $(\alpha,\beta)$ can be written as
\begin{align}\label{2.6}
\begin{split}
\pa_1\beta=-\fr{\gamma+1}{4}\beta^2+\fr{\gamma-3}{4}\alpha\beta+A_1(\alpha-\beta) +\fr{(\gamma-3)m}{r}\fr{u^2h^2}{c_{1}^2c_2}\beta, \\
\pa_2\alpha=-\fr{\gamma+1}{4}\alpha^2+\fr{\gamma-3}{4}\alpha\beta+A_2(\beta-\alpha) +\fr{(\gamma-3)m}{r}\fr{u^2h^2}{c_{2}^2c_1}\alpha.
\end{split}
\end{align}
In fact, to prove \eqref{2.6} by Lemma \ref{lemma_ric}, we simply use two equations
\begin{align*}
&\frac{c_2^2}{2}(\fr{\gamma-1}{2}u^2-h^2)-
c_2\bigg(\frac{\gamma-1}{4}u^3-\frac{1}{2}h^3
-\frac{\gamma-1}{4}u^2h+
\frac{1}{2}uh^2\bigg)-
hu c_1 (h+\frac{\gamma-1}{2}u) \\
=&\frac{c_2^2}{2}(\fr{\gamma-1}{2}u^2-h^2)-
\fr{c_1c_2}{2}(\frac{\gamma-1}{2}u^2+h^2)-
hu c_1 (h+\frac{\gamma-1}{2}u) \\
=&\frac{\gamma-1}{2}u^2c_2h -h^2c_2 u-
hu c_1 (h+\frac{\gamma-1}{2}u) \\
=&\frac{\gamma-1}{2}u^2h(c_2-c_1) -h^2u(c_2+c_1)
\\
=&(\gamma-3)u^2h^2,
\end{align*}
and
\begin{align*}
&\frac{c_1^2}{2}(\fr{\gamma-1}{2}u^2-h^2)-
c_1\bigg(\frac{\gamma-1}{4}u^3+\frac{1}{2}h^3
+\frac{\gamma-1}{4}u^2h+
\frac{1}{2}uh^2\bigg)-
hu c_2 (h-\frac{\gamma-1}{2}u) \\
=&\frac{c_1^2}{2}(\fr{\gamma-1}{2}u^2-h^2)-
\fr{c_1c_2}{2}(\frac{\gamma-1}{2}u^2+h^2)-
hu c_2 (h-\frac{\gamma-1}{2}u) \\
=&-\frac{\gamma-1}{2}u^2c_1h -h^2c_1 u-
hu c_2 (h-\frac{\gamma-1}{2}u) \\
=&\frac{\gamma-1}{2}u^2h(c_2-c_1) -h^2u(c_1+c_2)
\\
=&(\gamma-3)u^2h^2.
\end{align*}

Observing the simplified version Riccati system \eqref{2.6} provides the basis for the singularity formation result in this paper.

 Moreover, we introduce the weighted gradient variables
\begin{align}\label{a3}
\tilde\alpha=h^{\fr{\gamma-3}{2(\gamma-1)}}\alpha,\quad \tilde\beta=h^{\fr{\gamma-3}{2(\gamma-1)}}\beta.
\end{align}
By direct calculations, one gains
\begin{align}\label{a4}
\begin{split}
\pa_1\tilde{\beta}=&-\fr{\gamma+1}{4}h^{-\fr{\gamma-3}{2(\gamma-1)}}\tilde{\beta}^2 +\fr{mc_2}{2rc_{1}^2}\bigg(\fr{\gamma-1}{2}u^2-h^2\bigg)\tilde{\alpha} \\ & -\fr{m[(\gamma-3)u^2+(u+h)^2]}{2rc_1}\tilde{\beta}, \\
\pa_2\tilde{\alpha}=&-\fr{\gamma+1}{4}h^{-\fr{\gamma-3}{2(\gamma-1)}}\tilde{\alpha}^2 +\fr{mc_1}{2rc_{2}^2}\bigg(\fr{\gamma-1}{2}u^2-h^2\bigg)\tilde{\beta} \\ & -\fr{m[(\gamma-3)u^2+(u-h)^2]}{2rc_2}\tilde{\alpha}.
\end{split}
\end{align}
System \eqref{a4} can also derived by taking $\lambda=-\fr{\gamma-3}{2(\gamma-1)}$ in Lemma 6.1 in \cite{CEHS} and then doing a series of simplifications.

\begin{rem}
It is observed that the two equations in system \eqref{a4} are decoupled in its leading quadratic order term. Moreover, the coefficients of their first-order terms can be easily determined, which is crucial for our analysis below.
\end{rem}

It suggests by \eqref{2.5} that
\begin{align}\label{2.7}
h_r=\fr{\gamma-1}{4}\left(\alpha-\beta-\fr{2mu^2h}{rc_1c_2},\right)\quad u_r=\fr{\alpha+\beta}{2}+\fr{muh^2}{rc_1c_2},
\end{align}
from which and \eqref{2.1}, \eqref{a3}, we acquire the following equations
\begin{align}\label{2.8}
\pa_1h=&h_t+(u-h)h_r \notag \\
=&-hh_r-\fr{\gamma-1}{2}hu_r-\fr{\gamma-1}{2}\fr{muh}{r} \nonumber \\
=&-\fr{\gamma-1}{2}h\bigg(w_r+\fr{mu}{r}\bigg)\notag \\
=&-\fr{\gamma-1}{2}h\bigg(\alpha+\fr{mu^2}{rc_2}\bigg) \notag \\
=&-\fr{\gamma-1}{2}h^{\fr{\gamma+1}{2(\gamma-1)}}\bigg(\tilde\alpha +\fr{mu^2}{rc_2}h^{\fr{\gamma-3}{2(\gamma-1)}}\bigg),
\end{align}
and
\begin{align}\label{2.9}
\pa_1u=&u_t+(u-h)u_r \notag \\
=&-hu_r-\fr{2}{\gamma-1}h_r =-hw_r \nonumber \\
=&-h\bigg(\alpha-\fr{muh}{rc_2}\bigg) \notag \\
=&-h^{\fr{\gamma+1}{2(\gamma-1)}}\bigg(\tilde\alpha-\fr{muh}{rc_2}h^{\fr{\gamma-3}{2(\gamma-1)}}\bigg),
\end{align}
and
\begin{align}\label{2.10}
\pa_1c_2=&\pa_1\bigg(z+\fr{\gamma+1}{\gamma-1}h\bigg) \notag \\
=&\fr{muh}{r} +\fr{\gamma+1}{\gamma-1}\bigg(-\fr{\gamma-1}{2}hu_r-\fr{\gamma-1}{2}\fr{muh}{r}-hh_r\bigg) \notag \\
=&\fr{muh}{r} -\fr{\gamma+1}{2}h\bigg(w_r+\fr{mu}{r}\bigg) \notag \\
=&-\fr{\gamma+1}{2}h\bigg(\alpha+\fr{mu[(\gamma-1)u-2h]}{(\gamma+1)rc_2}\bigg) \notag \\
=&-\fr{\gamma+1}{2}h^{\fr{\gamma+1}{2(\gamma-1)}}\bigg(\tilde\alpha +\fr{mu[(\gamma-1)u-2h]}{(\gamma+1)rc_2}h^{\fr{\gamma-3}{2(\gamma-1)}}\bigg).
\end{align}

\section{Proof of the main result}\label{S3}

This section is devoted to showing the main result of the paper. The key technique is to construct a suitable invariant domain for variables $u, h, \tilde{\alpha}, \tilde{\beta}$ and $c_2$ at an appropriately small time.

Let $0<r_0<r_1<r_2$ be any three positive numbers. Assume that the initial data $(h(r,0),u(r,0))=(h_0(r), u_0(r))$ satisfy that
\begin{align}\label{3.1}
0<\underline{h}< h_0(r)< \overline{h}< \fr{1}{2}\overline{u},\quad -\underline{u}< u_0(r)< -\overline{u}<0,\quad \forall\ r\in[r_1,r_2],
\end{align}
for some positive constants $\underline{h}, \overline{h}, \underline{u}$ and $\overline{u}$. Denote
\begin{align}\label{3.2}
\begin{split}
\tilde{\alpha}_0(r)=&\bigg(u_{0}'(r)+h_{0}'(r)+\fr{mh_0(r)u_0(r)}{r[u_0(r)+h_0(r)]}\bigg) h_0(r)^{\fr{\gamma-3}{2(\gamma-1)}},\\ \tilde{\beta}_0(r)=&\bigg(u_{0}'(r)-h_{0}'(r)-\fr{mh_0(r)u_0(r)}{r[u_0(r)-h_0(r)]}\bigg) h_0(r)^{\fr{\gamma-3}{2(\gamma-1)}}.
\end{split}
\end{align}
We suppose that $(\tilde{\alpha}_0(r),\tilde{\beta}_0(r))$ satisfy
\begin{align}\label{3.3}
\underline{\alpha}< \tilde{\alpha}_0(r)< \overline{\alpha},\quad \tilde{\beta}_0(r)< -\overline{\beta},\quad \forall\ r\in[r_1,r_2],
\end{align}
for the positive constants $\underline{\alpha}, \overline{\alpha}$ and $\overline{\beta}$ fulfilling
\begin{align}\label{3.4}
\underline{\alpha}\geq \fr{2m(\underline{u}+\overline{h})}{r_0}\overline{h}^{\fr{\gamma-3}{2(\gamma-1)}},\quad \overline{\beta}> \fr{(\gamma+1)r_2\underline{\alpha}^2}{2m\underline{u}}\underline{h}^{-\fr{\gamma-3}{2(\gamma-1)}}.
\end{align}
Denote
$$
\alpha_*=\fr{2m(\underline{u}+\overline{h})}{r_0}\overline{h}^{\fr{\gamma-3}{2(\gamma-1)}},\quad \alpha^*=\bigg(\fr{\overline{h}}{\underline{h}}\bigg)^{\fr{\gamma+1}{2(\gamma-1)}}\overline{\alpha}.
$$
We further set
$$
\widetilde{T}=\min\bigg\{\fr{r_1-r_0}{\underline{u}+2\overline{h}},\ \fr{2}{(\gamma-1)\alpha^*}\underline{h}^{\fr{\gamma-3}{2(\gamma-1)}}\bigg\},
$$
such that
\begin{align}\label{3.5}
r_1-(\underline{u}+2\overline{h})\widetilde{T}\geq r_0,\quad \underline{h}^{-\fr{\gamma-3}{2(\gamma-1)}}\alpha^*\widetilde{T}\leq \fr{2}{\gamma-1}.
\end{align}

\begin{figure}[htbp]
\begin{center}
\includegraphics[scale=0.6]{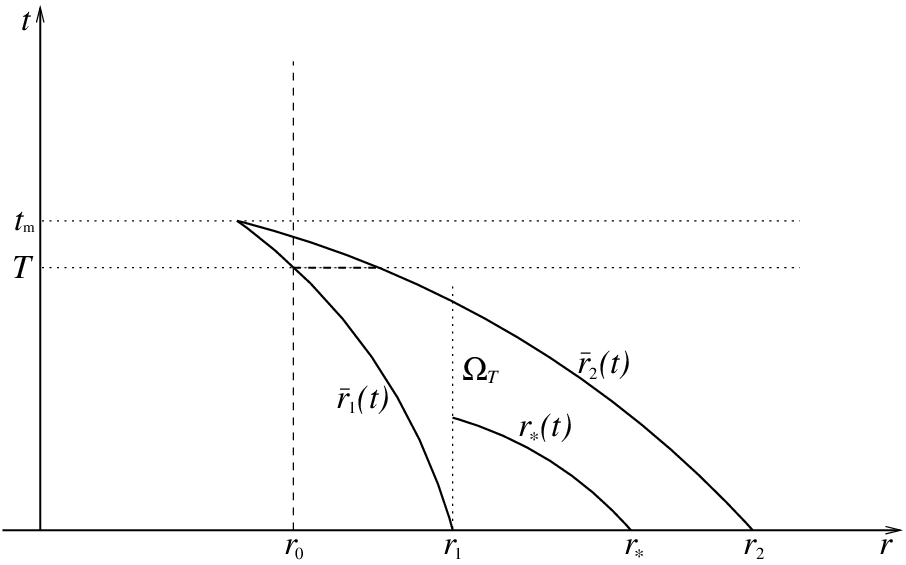}
\caption{\footnotesize The region $\Omega_T$.}
\label{fig1}
\end{center}
\end{figure}

Let $r=\bar{r}_2(t)$ and $r=\bar{r}_1(t)$ be the 1- and 2-characteristic curves through points $(r_2,0)$ and $(r_1,0)$, respectively, that is
\begin{align}\label{3.6}
\left\{
\begin{array}{l}
\fr{{\rm d}\bar{r}_1(t)}{{\rm d}t}=(u+h)(\bar{r}_1(t),t),\\
\bar{r}_1(0)=r_1,
\end{array}
\right.
\quad
\left\{
\begin{array}{l}
\fr{{\rm d}\bar{r}_2(t)}{{\rm d}t}=(u-h)(\bar{r}_2(t),t),\\
\bar{r}_2(0)=r_2.
\end{array}
\right.
\end{align}
We denote the intersection time of $r=\bar{r}_2(t)$ and $r=\bar{r}_1(t)$ by $t_m$ and set
\begin{align}\label{3.7}
\Omega_T=\{(r,t)|\ \bar{r}_1(t)\leq r\leq \bar{r}_2(t), 0\leq t\leq T\},
\end{align}
where $T\leq\min\{\widetilde{T},t_m\}$. See Fig. \ref{fig1} for illustration.

We have the following key lemma.
\begin{lem}\label{lem1}
Let $\gamma\geq3$. Assume that $(h,u)(r,t)$ is a smooth solution of \eqref{2.1} with initial data $(h(r,0),u(r,0))=(h_0(r), u_0(r))$. Let the initial data $(h_0(r), u_0(r))$ satisfy the conditions \eqref{3.1}, \eqref{3.3} and \eqref{3.4} on the interval $[r_1,r_2]$. Then there hold
\begin{align}\label{3.8}
\begin{split}
-(\underline{u}+\overline{h})< u(r,t)< -\overline{u},\quad e^{-1}\underline{h}< h(r,t)<\overline{h}, \\ c_2< -\fr{1}{2}\overline{u}, \quad
\alpha_*< \tilde{\alpha}(r,t)<\alpha^*,\quad \tilde{\beta}(r,t)< -\overline{\beta},
\end{split}
\quad \ \ \forall\ (r,t)\in\Omega_T.
\end{align}
\end{lem}
\begin{rem}
This lemma provides an invariant region for variables $(h, u, c_2, \tilde{\alpha}, \tilde{\beta})$ in $\Omega_T$. See Fig. \ref{fig2} for illustration.
\end{rem}

\begin{figure}[htbp]
\begin{center}
\begin{minipage}[t]{0.38\textwidth}
\includegraphics[scale=0.48]{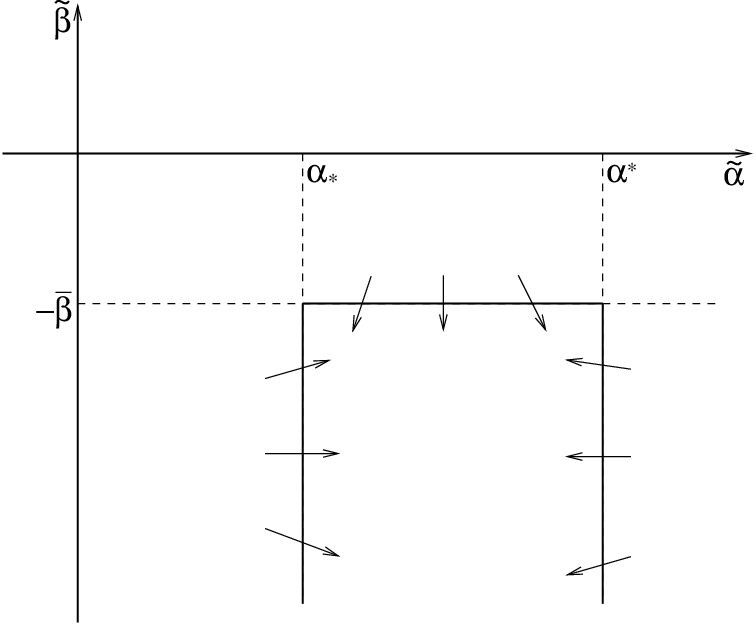}
{\center \qquad  \qquad \qquad \qquad  \  (a)}
\end{minipage}\quad
\begin{minipage}[t]{0.38\textwidth}
\includegraphics[scale=0.42]{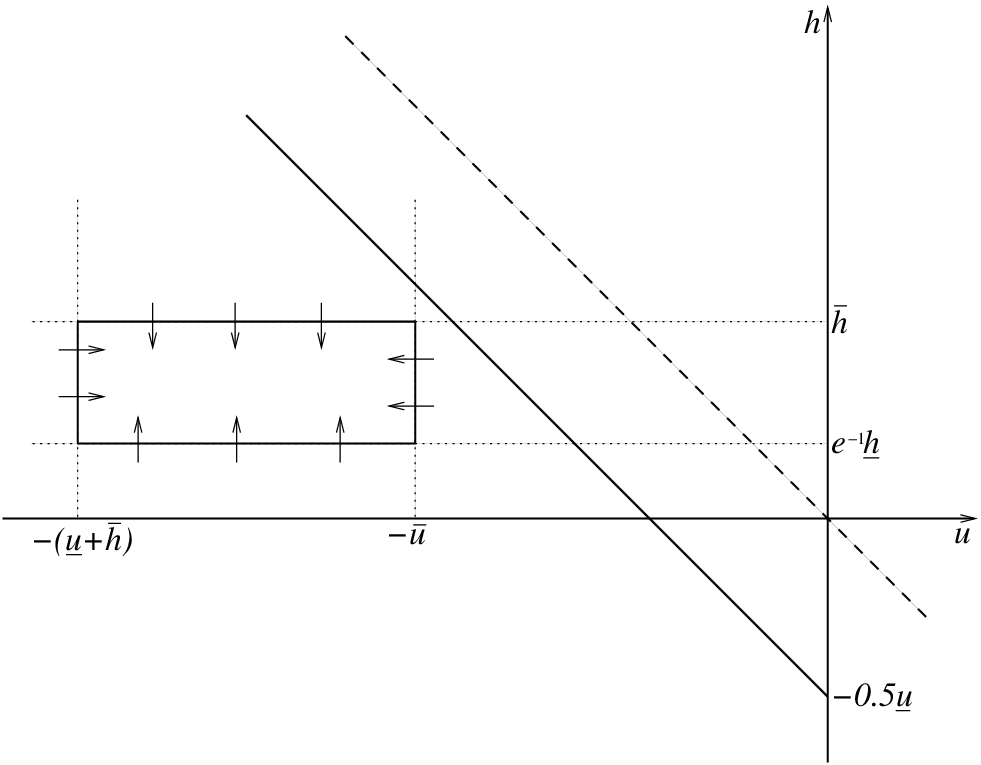}
{\center \qquad  \qquad \qquad \qquad \ (b)}
\end{minipage}
\caption{(a) The region of $(\tilde{\alpha},\tilde{\beta})$ for $(r,t)\in\Omega_T$; (b) The region of $(h,u)$ for $(r,t)\in\Omega_T$. }
\label{fig2}
\end{center}
\end{figure}

\begin{proof}
The proof of the lemma is based on the idea of contradiction.

(1) The estimates of $u$. Suppose that there exists a point $P(r^*,t^*)\in\Omega_T$ such that $u(r^*,t^*)=-\overline{u}$ and
\begin{align}\label{3.9}
\begin{split}
-(\underline{u}+\overline{h})\leq u(r,t)< -\overline{u},\quad e^{-1}\underline{h}\leq h(r,t)\leq\overline{h}, \\
c_2\leq -\fr{1}{2}\overline{u}, \quad
\alpha_*\leq\tilde{\alpha}(r,t)\leq\alpha^*,\quad \tilde{\beta}(r,t)\leq -\overline{\beta},
\end{split}
 \ \ \forall\ (r,t)\in\Omega_T\cap\{t<t^*\}.
\end{align}
This means that $u$ first touches the upper bound $-\overline{u}$ at point $P$. Thus one has
\begin{align}\label{3.10}
\pa_1u(P)\geq0.
\end{align}
We draw the 1-characteristic curve from the point $(r^*,t^*)$, denoted by $r_1(t)$, up to the line $t=0$ at a point $(\hat{r},0)$. Then we obtain by \eqref{3.5}
\begin{align}\label{3.11}
r^*=&\hat{r}+\int_{0}^{t^*}(u-h)(r_1(t),t)\ {\rm d}t  \nonumber \\
\geq & r_1-\int_{0}^{T}(\underline{u}+2\overline{h})\ {\rm d}t =r_1-(\underline{u}+2\overline{h})T\geq r_0.
\end{align}
Then it follows by \eqref{3.9} that
\begin{align}\label{3.12}
\bigg(\tilde{\alpha}-\fr{muh}{rc_2}h^{\fr{\gamma-3}{2(\gamma-1)}}\bigg)(P)\geq &\bigg(\fr{2m(\underline{u}+\overline{h})}{r_0} -\fr{m\overline{h}}{r_0}\cdot \fr{u}{c_2}(P)\bigg)\overline{h}^{\fr{\gamma-3}{2(\gamma-1)}} \nonumber \\
\geq &\bigg(\fr{2m(\underline{u}+\overline{h})}{r_0} -\fr{2m\overline{h}}{r_0} \bigg)\overline{h}^{\fr{\gamma-3}{2(\gamma-1)}} \notag \\ =&\fr{2m\underline{u}}{r_0}\overline{h}^{\fr{\gamma-3}{2(\gamma-1)}}>0.
\end{align}
Here we used the result $1\leq\fr{u}{c_2}\leq2$ by the fact $c_2=\fr{1}{2}u+(\fr{1}{2}u+h)\leq \fr{1}{2}u$.
Combining \eqref{2.9} and \eqref{3.12} yields
\begin{align*}
\pa_1u(P)=-h^{\fr{\gamma+1}{2(\gamma-1)}}\bigg(\tilde\alpha -\fr{muh}{rc_2}h^{\fr{\gamma-3}{2(\gamma-1)}}\bigg)(P)<0,
\end{align*}
which contradicts \eqref{3.10}.

To obtain the lower bound of $u$, we integrate \eqref{2.9} along the 1-characteristic curve $r_1(t)$ to acquire
\begin{align}\label{3.13}
u(r,t)=& u_0(r_1(0))-\int_{0}^t h^{\fr{\gamma+1}{2(\gamma-1)}}\bigg(\tilde\alpha -\fr{muh}{rc_2}h^{\fr{\gamma-3}{2(\gamma-1)}}\bigg)(r_1(s),s)\ {\rm d}s \nonumber \\
> &-\underline{u}-\int_{0}^t\bigg(h\cdot h^{-\fr{\gamma-3}{2(\gamma-1)}}\tilde\alpha\bigg)(r_1(s),s)\ {\rm d}s \nonumber \\ >&-\underline{u}-\int_{0}^T\overline{h}\cdot\underline{h}^{-\fr{\gamma-3}{2(\gamma-1)}}  \alpha^*\ {\rm d}s \geq -(\underline{u}+\overline{h}),
\end{align}
by the fact $\underline{h}^{-\fr{\gamma-3}{2(\gamma-1)}}\alpha^*T\leq\fr{2}{\gamma-1}\leq1$.

(2) The estimates of $h$.
Suppose that there exists a point $P(r^*,t^*)\in\Omega_T$ such that $h(r^*,t^*)=\overline{h}$ and
\begin{align}\label{3.14}
\begin{split}
-(\underline{u}+\overline{h})\leq u(r,t)\leq-\overline{u},\quad e^{-1}\underline{h}\leq h(r,t)<\overline{h}, \\
c_2\leq -\fr{1}{2}\overline{u},\quad \alpha_*\leq\tilde{\alpha}(r,t)\leq\alpha^*,\quad \tilde{\beta}(r,t)\leq -\overline{\beta},
\end{split}
\quad \forall\ (r,t)\in\Omega_T\cap\{t<t^*\}.
\end{align}
Then
\begin{align}\label{3.15}
\pa_1h(P)\geq0.
\end{align}
On the other hand, we find by \eqref{3.14} that
\begin{align*}
\bigg(\tilde\alpha +\fr{mu^2}{rc_2}h^{\fr{\gamma-3}{2(\gamma-1)}}\bigg)(P) >\alpha_* -\fr{2m}{r_0}(\underline{u}+\overline{h})\overline{h}^{\fr{\gamma-3}{2(\gamma-1)}}=0,
\end{align*}
from which and \eqref{2.8} one has
\begin{align*}
\pa_1h(P)=-\fr{\gamma-1}{2}h^{\fr{\gamma+1}{2(\gamma-1)}}\bigg(\tilde\alpha +\fr{mu^2}{rc_2}h^{\fr{\gamma-3}{2(\gamma-1)}}\bigg)(P) <0,
\end{align*}
a contradiction.

Now we rewrite \eqref{2.8} as the following form
\begin{align*}
\pa_1\ln h=&-\fr{\gamma-1}{2}h^{-\fr{\gamma-3}{2(\gamma-1)}}\bigg(\tilde\alpha +\fr{mu^2}{rc_2}h^{\fr{\gamma-3}{2(\gamma-1)}}\bigg) \\
>& -\fr{\gamma-1}{2}h^{-\fr{\gamma-3}{2(\gamma-1)}}\tilde\alpha.
\end{align*}
Integrating the above along the 1-characteristic curve $r_1(t)$ leads to
\begin{align}\label{3.16}
h(r,t)\geq&h_0(r_1(0))\exp\bigg\{\int_{0}^t-\fr{\gamma-1}{2}h^{-\fr{\gamma-3}{2(\gamma-1)}}\tilde\alpha (r_1(t),t)\ {\rm d}t\bigg\} \nonumber \\
>& \underline{h}\exp\bigg\{-\fr{\gamma-1}{2}\underline{h}^{-\fr{\gamma-3}{2(\gamma-1)}}\alpha^* T\bigg\}\geq e^{-1}\underline{h}.
\end{align}

(3) The estimates of $c_2$. It is first suggests that
\begin{align*}
\bigg|\fr{mu[(\gamma-1)u-2h]}{(\gamma+1)rc_2} \bigg| =&\bigg|\fr{mu}{rc_2}\cdot\bigg(\fr{\gamma-1}{\gamma+1}u-\fr{2}{\gamma+1}h\bigg)\bigg| \\
\leq & \fr{2m\underline{u}}{r_0}(\underline{u}+\overline{h}),
\end{align*}
from which and \eqref{2.10} we see that
\begin{align}\label{3.17}
\pa_1c_2=&-\fr{\gamma+1}{2}h^{\fr{\gamma+1}{2(\gamma-1)}}\bigg(\tilde\alpha +\fr{mu[(\gamma-1)u-2h]}{(\gamma+1)rc_2}h^{\fr{\gamma-3}{2(\gamma-1)}}\bigg) \notag \\
<& -\fr{\gamma+1}{2}h^{\fr{\gamma+1}{2(\gamma-1)}}\bigg(\alpha_* -\bigg|\fr{mu[(\gamma-1)u-2h]}{(\gamma+1)rc_2}h^{\fr{\gamma-3}{2(\gamma-1)}}\bigg|\bigg) \notag \\
<& -\fr{\gamma+1}{2}h^{\fr{\gamma+1}{2(\gamma-1)}}\bigg(\alpha_* - \fr{2m\underline{u}}{r_0}(\underline{u}+\overline{h})\overline{h}^{\fr{\gamma-3}{2(\gamma-1)}}\bigg)
=0, \quad {\rm in\ the\ region}\ \Omega_T,
\end{align}
which implies that for any $(r,t)\in\Omega_T$
\begin{align}\label{3.18}
c_2(r,t)<\min_{r\in[r_1,r_2]}(u_0(r)+h_0(r))<-\overline{u}+\overline{h}<-\fr{1}{2}\overline{u}.
\end{align}

(4) The estimates of $\tilde{\alpha}$. Suppose that there exists a point $P(r^*,t^*)\in\Omega_T$ such that $\tilde{\alpha}(r^*,t^*)=\alpha_*$ and
\begin{align}\label{3.19}
\begin{split}
-(\underline{u}+\overline{h})\leq u(r,t)\leq-\overline{u},\quad e^{-1}\underline{h}\leq h(r,t)\leq\overline{h}, \\
c_2\leq -\fr{1}{2}\overline{u},\quad \alpha_*<\tilde{\alpha}(r,t)\leq\alpha^*,\quad \tilde{\beta}(r,t)\leq -\overline{\beta},
\end{split}
\ \ \forall\ (r,t)\in\Omega_T\cap\{t<t^*\}.
\end{align}
Then
\begin{align}\label{3.20}
\pa_2\tilde{\alpha}(P)\leq0.
\end{align}
On the other hand, we recall \eqref{a4} and employ \eqref{3.19} and \eqref{3.4} to achieve
\begin{align}\label{3.21}
\pa_2\tilde{\alpha}(P)=&\bigg\{ -\fr{\gamma+1}{4}h^{-\fr{\gamma-3}{2(\gamma-1)}}\tilde{\alpha}^2 +\fr{mc_1}{2rc_{2}^2}\bigg(\fr{\gamma-1}{2}u^2-h^2\bigg)\tilde{\beta} \\ & -\fr{m[(\gamma-3)u^2+(u-h)^2]}{2rc_2}\tilde{\alpha} \bigg\}(P) \nonumber \\ >&-\fr{\gamma+1}{4}h(P)^{-\fr{\gamma-3}{2(\gamma-1)}} \alpha_*^2 +\fr{mc_1}{2rc_{2}^2}\bigg(\fr{\gamma-1}{2}u^2-h^2\bigg)(P)\tilde{\beta}(P)  \nonumber \\
>& \fr{-mc_1}{2rc_{2}^2}(u^2-h^2)(P)\overline{\beta} -\fr{\gamma+1}{4}\underline{h}^{-\fr{\gamma-3}{2(\gamma-1)}} \alpha_*^2 \nonumber \\
>& \fr{m\underline{u}}{2r_2}\overline{\beta} -\fr{\gamma+1}{4}\underline{h}^{-\fr{\gamma-3}{2(\gamma-1)}} \alpha_*^2 \notag \\
\geq& (\underline{\alpha}^2-\alpha_*^2)\underline{h}^{-\fr{\gamma-3}{2(\gamma-1)}}\geq0,
\end{align}
which leads to a contradiction with \eqref{3.20}.

To establish the upper bound of $\tilde{\alpha}$, we first derive by \eqref{2.6} and \eqref{2.1}
\begin{align}\label{3.22}
\pa_2\bigg(h^{-\fr{\gamma+1}{2(\gamma-1)}}\tilde{\alpha}\bigg) =&\bigg\{\fr{\gamma+1}{4}h^{-\fr{\gamma-3}{2(\gamma-1)}}\tilde\alpha +\fr{mc_1}{2rc_{2}^2}\bigg(\fr{\gamma-1}{2}u^2-h^2\bigg)\bigg\}h^{-\fr{\gamma+1}{2(\gamma-1)}} (\tilde{\beta}-\tilde{\alpha}) \notag \\
&+\fr{mu^2[c_{2}^2+(\gamma-3)h^2]}{rc_1c_{2}^2} h^{-\fr{\gamma+1}{2(\gamma-1)}}\tilde{\alpha}
\end{align}
Note that
\begin{align}\label{3.23}
&\fr{\gamma+1}{4}h^{-\fr{\gamma-3}{2(\gamma-1)}}\tilde\alpha +\fr{mc_1}{2rc_{2}^2}\bigg(\fr{\gamma-1}{2}u^2-h^2\bigg) \notag \\
> & \overline{h}^{-\fr{\gamma-3}{2(\gamma-1)}}\alpha_* +\fr{mc_{1}^2}{2rc_{2}} \notag \\ >& \fr{2m(\underline{u}+\overline{h})}{r_0} -\fr{3m(\underline{u}+\overline{h})}{2r_0}>0,
\end{align}
where we used the fact
$$
1<\fr{c_1}{c_2}=\fr{u-h}{u+h}\leq \fr{u+\fr{1}{2}u}{u-\fr{1}{2}u}=3.
$$
One combines \eqref{3.22} and \eqref{3.23} to obtain
\begin{align}\label{3.24}
\pa_2\bigg(h^{-\fr{\gamma+1}{2(\gamma-1)}}\tilde{\alpha}\bigg)<0,\quad \forall\ (r,t)\in\Omega_T.
\end{align}
Thus we have
$$
h^{-\fr{\gamma+1}{2(\gamma-1)}}\tilde{\alpha} <h^{-\fr{\gamma+1}{2(\gamma-1)}}\tilde{\alpha}(r_2(0),0)\leq \underline{h}^{-\fr{\gamma+1}{2(\gamma-1)}}\overline{\alpha},
$$
which gives
\begin{align}\label{3.25}
\tilde{\alpha}<h^{\fr{\gamma+1}{2(\gamma-1)}} \cdot\underline{h}^{-\fr{\gamma+1}{2(\gamma-1)}}\overline{\alpha} <\bigg(\fr{\overline{h}}{\underline{h}}\bigg)^{\fr{\gamma+1}{2(\gamma-1)}}\overline{\alpha}=\alpha^*.
\end{align}
Here $r_2(t)$ is the 2-characteristic curve from the point $(r,t)$ up to the line $t=0$.

(5) The estimates of $\tilde{\beta}$. We note that each of the three terms at the right-hand of $\pa_1\tilde{\beta}$ in \eqref{a4} is negative. Thus one has
\begin{align*}
\pa_1\tilde{\beta}<0,
\end{align*}
which indicates that
$$
\tilde{\beta}(r,t)<-\overline{\beta},\quad \forall\ (r,t)\in\Omega_T.
$$
The proof of the lemma is complete.
\end{proof}

Based on Lemma \ref{lem1}, we have the following theorem.
\begin{thm}\label{thm1}
Let the conditions in Lemma \ref{lem1} be satisfied. We further assume that
\begin{align}\label{3.26}
\tilde{\beta}_0(r_*)\leq -\max\bigg\{\overline{\beta}\cdot\overline{h}^{-\fr{\gamma-3}{2(\gamma-1)}}, \fr{1}{T}, \fr{\underline{u}+2\overline{h}}{r_*-r_1}\bigg\}:=-N(T,r_*),
\end{align}
for some number $r_*\in(r_1,r_2)$. Then singularity forms before time $\fr{1}{N(T,r_*)}\leq T$.
\end{thm}
\begin{proof}
According to Lemma \ref{lem1}, we know that the smooth solution, if it exists, satisfies the estimates in \eqref{3.8} in the region $\Omega_T$. From the point $(r_*,0)$, we draw the 1-characteristic curve $r=r_*(t)$. Then
\begin{align}
r_*(t)=&r_*+\int_{0}^t(u-h)(r_*(s),s)\ {\rm d}s \nonumber \\
\geq & r_*-(\underline{u}+2\overline{h})t, \nonumber
\end{align}
from which one sees that
$$
r_*(t)\geq r_1,\quad \forall\ t\leq \fr{1}{N(T,r_*)}.
$$
This implies that the curve $r=r_*(t)$ is in the region $\Omega_T$ before the time $\fr{1}{N(T,r_*)}$. Thus along the curve $r=r_*(t)$, we have by \eqref{3.8}
$$
\fr{mc_2}{2rc_{1}^2}\bigg(\fr{\gamma-1}{2}u^2-h^2\bigg)\tilde{\alpha}<0,\quad -\fr{m[(\gamma-3)u^2+(u-h)^2]}{2rc_2}\tilde{\alpha}<0,
$$
from which and \eqref{a4} one obtains
\begin{align}\label{3.27}
\pa_1\tilde{\beta}(r_*(t),t)<&-\fr{\gamma+1}{4}h^{-\fr{\gamma-3}{2(\gamma-1)}}\tilde{\beta}^2(r_*(t),t) \notag \\
<& -\overline{h}^{-\fr{\gamma-3}{2(\gamma-1)}}\tilde{\beta}^2(r_*(t),t),\quad {\rm for}
\quad t\leq \fr{1}{N(T,r_*)}.
\end{align}
Integrate \eqref{3.27} yields
\begin{align}\label{3.28}
-\tilde{\beta}(r_*(t),t)>\fr{-\beta_0(r_*)}{1+\beta_0(r_*)\overline{h}^{-\fr{\gamma-3}{2(\gamma-1)}}t},
\end{align}
which indicates that the smooth solution forms singularity before time $\fr{1}{N(T,r_*)}$. The proof of the theorem is complete.
\end{proof}

\section*{Acknowledgements}

The first and second authors were partially supported by National Science Foundation (DMS-2008504, DMS-2306258), the third author was partially supported by National Natural
Science Foundation of China (12171130) and Natural Science Foundation of
Zhejiang province of China (LMS25A010014), the fourth author was partially supported by National Science Foundation (DMS-2206218).

\section*{Data Availability Statement}
No data was used for the research described in the paper.

\section*{Statements and Declarations}

The authors declare that they have no conflict of interest.

\end{document}